\documentclass{amsart}% option [reqno] puts equat numb to the right
\pdfoutput=1

\usepackage[mathscr]{eucal}% So-called Euler fonts, allows \mathscr
\usepackage{amssymb}% allows special arrows like >-> e.g.
\usepackage{stmaryrd}
\usepackage[usenames,dvipsnames]{xcolor}
\usepackage{xspace}%allows smart spacing after period
\usepackage{amsthm}% allows theorem* , e.g.
\usepackage{bbold}% allows \unit \mathbb{1}
\usepackage{enumitem}% allows easy change of label
\setlist{leftmargin=*, wide, labelindent=0pt}
\setlist[enumerate]{label*=\rm(\alph*),ref=\alph*}

\usepackage{amsmath}% allows pmatrix
\usepackage{mathdots}% allows diagonal dots

\numberwithin{equation}{section}% makes equat numb contain the section
\usepackage[all]{xy}
\SelectTips{cm}{}
\newdir{ >}{{}*!/-10pt/\dir{>}}

\usepackage{xr-hyper}

\usepackage[colorlinks=true,linkcolor={Black},citecolor={Black},urlcolor={Black}]{hyperref}
\definecolor{azure(colorwheel)}{rgb}{0.0, 0.5, 1.0}
\definecolor{amber}{rgb}{1.0, 0.49, 0.0}

\usepackage[nameinlink]{cleveref}% allows references via \cref and \Cref which automatically repeat the name (Thm, Prop, etc) of the ref.

\crefname{Thm}{Theorem}{Theorems}
\crefname{Rem}{Remark}{Remarks}
\crefname{Prop}{Proposition}{Propositions}
\crefname{Cor}{Corollary}{Corollaries}
\crefname{Cons}{Construction}{Constructions}
\crefname{Exa}{Example}{Examples}
\crefname{Lem}{Lemma}{Lemmas}
\crefname{Rec}{Recollection}{Recollections}

\swapnumbers %pour que le numero apparaisse devant le theoreme
\newtheorem{Cor}[equation]{Corollary}
\newtheorem{Lem}[equation]{Lemma}
\newtheorem{Prop}[equation]{Proposition}
\newtheorem{Thm}[equation]{Theorem}

\theoremstyle{remark}
\newtheorem{Def}[equation]{Definition}
\newtheorem{Not}[equation]{Notation}
\newtheorem{Exa}[equation]{Example}
\newtheorem{Hyp}[equation]{Hypothesis}
\newtheorem{Rem}[equation]{Remark}
\newtheorem{Rec}[equation]{Recollection}
\newtheorem*{Ack}{Acknowledgements}

\newcommand{\nc}{\newcommand}
\nc{\dmo}{\DeclareMathOperator}

\dmo{\con}{con}
\dmo{\cone}{cone}
\dmo{\Der}{D}% ground notation for derived categories
\dmo{\DPerm}{DPerm}
\dmo{\ev}{ev}
\dmo{\Hm}{H}% (co)homology
\dmo{\id}{id}
\dmo{\Img}{Im}
\dmo{\Infl}{Infl}
\dmo{\Ker}{Ker}
\dmo{\Res}{Res}
\dmo{\SH}{SH}% ground name for cat of compact spectra
\dmo{\smallb}{b}% ground exponent for ``bounded''
\dmo{\smallperf}{perf}% ground exponent for ``perfect''
\dmo{\Spc}{Spc}
\dmo{\Spec}{Spec}
\dmo{\Spech}{\Spec^{h}}
\dmo{\subname}{Sub}
\dmo{\supp}{supp}

\nc{\adj}{\dashv}
\nc{\aka}{{a.\,k.\,a.}\ }
\nc{\calU}{\mathcal{U}}
\nc{\cat}[1]{\mathscr{#1}}%or: \nc{\cat}[1]{\mathcal{#1}}
\nc{\cC}{\cat{C}}
\nc{\cJ}{\cat{J}}
\nc{\cK}{\cat{K}}
\nc{\cL}{\cat{L}}
\nc{\colim}{\mathop{\mathrm{colim}}}
\nc{\cP}{\cat{P}}
\nc{\cT}{\cat{T}}
\nc{\Db}{\Der_{\smallb}}% derived bounded category%\scriptscriptstyle
\nc{\Dperf}{\Der_{\smallperf}}% derived category of perfect compl%\scriptscriptstyle
\nc{\eg}{{\sl e.g.}\@\xspace}
\nc{\gm}{\mathfrak{m}}% prime m
\nc{\gp}{\mathfrak{p}}% prime p
\nc{\ideal}[1]{\langle #1\rangle}
\nc{\ie}{{\sl i.e.}\@\xspace}
\nc{\ihom}{{\mathsf{hom}}} %{{\underline{\hom}}}
\nc{\inv}{^{-1}}
\nc{\isoto}{\overset{\sim}{\,\to\,}}
\nc{\loccit}{{\sl loc.\ cit.}\xspace}
\nc{\Mid}{\,\big|\,}
\nc{\normaleq}{\trianglelefteqslant}
\nc{\potimes}[1]{^{\otimes #1}}% tensor power
\nc{\qcO}{\mathcal{QO}}
\nc{\SET}[2]{\big\{\,#1\Mid#2\,\big\}}
\nc{\sminus}{\smallsetminus}
\nc{\SpcK}{\Spc(\cK)}% most used
\nc{\SpSp}{\mathsf{Spec}}
\nc{\To}{\Rightarrow}
\nc{\too}{\mathop{\longrightarrow}\limits}
\nc{\unit}{\mathbb{1}}% unit for \otimes
\nc{\WGH}{{G}/\!\!/{H}}
\nc{\wX}{\overline{X}{}^{\calU}}
\nc{\xto}[1]{\xrightarrow{#1}}

\date{2025 March 19}%\today}

\author{Paul Balmer}
\address{Paul Balmer, UCLA Mathematics Department, Los Angeles, CA 90095, USA}
\email{balmer@math.ucla.edu}
\urladdr{https://www.math.ucla.edu/~balmer}

\author{Martin Gallauer}
\address{Martin Gallauer, Mathematics Institute, University of Warwick}
\email{martin.gallauer@warwick.ac.uk}
\urladdr{https://mgallauer.warwick.ac.uk}

\hypersetup{pdfauthor={Paul Balmer and Martin Gallauer},pdftitle={Patch-density in tensor-triangular geometry}}

\begin{document}

%------------------------------------------------------------------------------

\title{Patch-density in tensor-triangular geometry}

\begin{abstract}
The spectrum of a tensor-triangulated category carries a compact Hausdorff topology, called the constructible topology, also known as the patch topology.
We prove that patch-dense subsets detect tt-ideals and we prove that any infinite family of tt-functors that detects nilpotence provides such a patch-dense subset.
We review several applications and examples in algebra, in topology and in the representation theory of profinite groups.
\end{abstract}

\subjclass[2020]{18F99}
\keywords{constructible topology, patch topology, tensor-triangular geometry, density, detection of nilpotence}

\maketitle

\section{Introduction}
\label{sec:intro}

%------------------------------------------------------------------------------

In this short note we sharpen our understanding of two fundamental themes of tensor-triangular geometry: the classification of thick tensor-ideals and the detection of tensor-nilpotence.
The origins of the subject can be traced back to the Nilpotence Theorem of Devinatz--Hopkins--Smith~\cite{devinatz-hopkins-smith:chromatic} in topology.
Using Morava $K$-theories~$K(n)$ at a prime~$p$, they prove that a morphism~$f\colon k\to L$ in the $p$-local stable homotopy category~$\SH_{(p)}$, with finite source~$k$, must be $\otimes$-nilpotent if $K(n)(f)=0$ for all~$0\le n\le \infty$; here, $K(\infty)$ means mod-$p$ homology.
This theorem led to a classification of the thick subcategories of finite $p$-local spectra in~\cite{hopkins-smith:chromatic}, as being exactly the so-called `chromatic' tower:
\begin{equation}
\label{eq:SH-classification}%
\SH^c_{(p)}=\cC_0\supsetneq \cC_1\supsetneq\cdots\supsetneq\ \cC_n=\Ker\big(K(n-1)_*\big)\ \supsetneq\cdots\supsetneq\cC_\infty=0
\end{equation}
Already in this initial example, we can point to the germ of what we wish to discuss.
On the one hand, in the Nilpotence Theorem, if the morphism $f\colon k\to \ell$ also has finite target~$\ell$ then we do not need to know that $K(\infty)(f)=0$ to conclude that $f$ is $\otimes$-nilpotent.
On the other hand, in the chromatic tower~\eqref{eq:SH-classification}, the smallest subcategory~$\cC_{\infty}$ is not really `seen' by any finite object: If $k\in\SH^c_{(p)}$ belongs to all~$\cC_n$ for~$n<\infty$ then it belongs to $\cC_{\infty}$ as well. In other words, there is no finite object that has infinite chromatic level.
In both cases, the `stuff at~$\infty$' seems somewhat irrelevant.
To explain the parallels between these two phenomena let us remind the reader of some elementary tt-geometry.

Let $\cK$ be an essentially small rigid tt-category, such as $\SH^c_{(p)}$ above.
To this, we can associate a space~$\Spc(\cK)$, called the spectrum, that affords the universal support theory $\supp(k)\subseteq\SpcK$ for~$\cK$, see~\cite{balmer:spectrum}.
A support theory for~$\cK$ on a space~$X$ consists of closed subsets $\sigma(k)\subseteq X$ for each object $k\in \cK$ that behave in a predictable manner as one operates on~$k$ through the tensor triangulated structure.
The supports in $\SpcK$ induce a classification of thick tensor-ideals:
\[
\{\,\text{tt-ideals in }\cK\,\}\xrightarrow[\sim]{\ \supp\ }\{\,\text{Thomason subsets of }\Spc(\cK)\,\}.
\]
Being universal means that any support theory $(X,\sigma)$ is classified by a continuous map $\phi\colon X\to\Spc(\cK)$ such that $\sigma(k)=\phi\inv(\supp(k))$ for all $k\in\cK$.
Of particular importance among all support theories~$(X,\sigma)$ are those that \emph{distinguish supports}, in the sense that $\sigma(k)=\sigma(\ell)$ forces $\supp(k)=\supp(\ell)$.
This means that the classifying map $\phi \colon X\to \SpcK$ might not be a homeomorphism but the supports in~$X$ are a fine enough invariant to distinguish tt-ideals.
Let us rephrase this property using point-set topology.

Recall that the \emph{constructible} topology on~$\SpcK$ -- \aka the \emph{patch} topology -- is generated by the quasi-compact opens and their complements. (\Cref{Rec:patch-topology}.)
It is easy to see that a support theory $(X,\sigma)$ distinguishes supports if and only if the image of the classifying map $\phi\colon X\to\Spc(\cK)$ is \emph{patch-dense}, \ie dense for the constructible topology.
For example, any patch-dense subset $X\subseteq \Spc(\cK)$ equipped with $\sigma(k):=X\cap\supp(k)$ distinguishes supports.
We can now state the main result of this paper, proved in \Cref{sec:detection}:
\begin{Thm}
\label{Thm:intro}
Let $\cK$ be an essentially small rigid tt-category and consider a family $\{F_i\colon \cK\to \cL_i\}_{i\in I}$ of tt-functors.
The following are equivalent:
\begin{enumerate}[label=\rm(\roman*), ref=\rm(\roman*)]
\item The tt-functors $F_i\colon \cK\to\cL_i$ jointly detect $\otimes$-nilpotence of morphisms that are $\otimes$-nilpotent on their cones.
\smallbreak
\item The subset $\cup_{i}\Img(\Spc(F_i))\subseteq\SpcK$ is patch-dense.
\smallbreak
\item[\rm(ii')]The maps $\Spc(F_i)\colon \Spc(\cL_i)\to\Spc(\cK)$ jointly distinguish supports.
\end{enumerate}
\end{Thm}

For a single tt-functor $F\colon \cK\to\cL$ this theorem recovers the surjectivity result of~\cite{balmer:surjectivity}; see \Cref{Rmk:surjective-singleton-family}.
Our proof is a generalization of that proof.

Let us go back to the example at the start of this introduction.
The classification in the chromatic tower~\eqref{eq:SH-classification} translates
into $\Spc(\SH^c_{(p)})$ being the space
\begin{equation}
\label{eq:Spc-SH}%
\cC_1\leadsto\cC_2\leadsto\cdots\leadsto\cC_{n}\leadsto\cC_{n+1}\leadsto\cdots\leadsto\cC_\infty
\end{equation}
in which the closed subsets are precisely the subsets closed under specialization ($\leadsto$ going towards the right in the above picture).
The observations made earlier about the irrelevance of $K(\infty)$ and of~$\cC_{\infty}$ are equivalent to the patch-density of the complement $X=\Spc(\SH^c_{(p)})\sminus\{\cC_{\infty}\}=\SET{\cC_{n}}{1\le n<\infty}$ of the closed point at infinity.
In fact, this $X$ is the only patch-dense proper subset of~$\Spc(\SH^c_{(p)})$.

\smallbreak
A second goal of this article is to identify situations where patch-density occurs.
We emphasize the following implication that is particularly interesting:
\begin{Cor}
If $\cK$ is rigid and a family of tt-functors $\{F_i\colon\cK\to\cL_i\}_{i\in I}$ jointly detects $\otimes$-nilpotence then the subset $\cup_i\Img(\Spc(F_i))\subseteq\Spc(\cK)$ is patch-dense.
\end{Cor}
There are numerous instances in the recent literature where the space $\Spc(\cK)$ is difficult to describe explicitly but where it is much easier to describe some natural patch-dense subset.
We explain in \Cref{Thm:reconstruction} how to recover $\SpcK$ from such a patch-dense subset~$X$ together with the restricted support theory~$\supp_{X}$ on~$X$, defined by $\supp_{X}(k)=X\cap \supp(k)$ for all~$k\in\cK$.
Several examples of these phenomena will be given in \Cref{sec:examples}.
We also point out the recent paper by G\'omez~\cite{Gomez:fiberwise} where families of tt-functors as above are used to study stratification of `big' tt-categories.

\begin{Ack}
We thank Scott Balchin and Tobias Barthel for interesting discussions during the preparation of this work.
\end{Ack}

\pagebreak

%------------------------------------------------------------------------------

\section{Patch-density}
\label{sec:density}

\begin{Not}
We denote by $\SpSp$ the category of spectral spaces in the sense of Hochster~\cite{hochster:prime-ideal-structure} with continuous spectral maps as morphisms.
\end{Not}

\begin{Rec}
\label{Rec:patch-topology}
Let $X$ be a spectral space, like for instance~$X=\SpcK$.
The collection $\qcO(X)$ of all quasi-compact open subsets is an open basis of~$X$.
Their complements form an open (!) basis of the so-called dual topology~$X^*$.
We also speak of \emph{Thomason subset} for an open in~$X^*$.
The smallest topology on the set~$X$ that contains both the original and the dual topology is called the \emph{constructible} (or \emph{patch}) topology, and is denoted~$X_{\con}$.
It is generated by quasi-compact opens in~$X$ and their complements, and is always a Boolean space (hence compact Hausdorff).
To be precise, a subset of~$X$ is called \emph{constructible} if it can be obtained by finite union and finite intersection from the quasi-compact opens of~$X$ and their complements.
A constructible-open is an arbitrary union of constructibles; a constructible-closed (or \emph{proconstructible}) is an arbitrary intersection of constructibles.
For more on this see~\cite[\S\,1.3]{dst:spectral-spaces}.
Note that a subset $D\subset X$ is \emph{patch-dense}, \ie dense in~$X_{\con}$, if and only if $D$ meets every non-empty constructible in~$X$.
Our textbook reference~\cite{dst:spectral-spaces} uses `patch dense' and `constructibly dense' interchangeably.
\end{Rec}

\begin{Rem}
\label{Rem:cons-SpcK}%
Let $\cK$ be an essentially small rigid tt-category and consider the spectral space $\Spc(\cK)=\{\cP\subsetneq\cK\textrm{ prime tt-ideal}\}$.
Each object $k\in \cK$ comes with a closed subset $\supp(k)=\SET{\cP}{k\notin\cP}$.
Its complement $U(k)=\supp(k)^c$ is quasi-compact, and all quasi-compact opens are of this form.
See~\cite{balmer:spectrum}.

Let us write $C(k,\ell)=\supp(k)\cap \supp(\ell)^c=\SET{\cP\in\SpcK}{k\notin \cP\textrm{ and }\ell\in \cP}$ for every~$k,\ell\in\cK$.
It follows from $C(k,\ell)\cap C(k',\ell')=C(k\otimes k',\ell\oplus \ell')$ that these constructible subsets $C(k,\ell)$ form a basis of the constructible topology on~$\SpcK$.
Thus a subset $D\subseteq \SpcK$ is patch-dense if and only if $D$ meets $C(k,\ell)$ for every $k,\ell\in\cK$ such that $\supp(k)\not\subseteq\supp(\ell)$.
\end{Rem}

\begin{Def}
\label{Def:separate-supports}%
A family of maps $\phi_i\colon Y_i\to \Spc(\cK)$ is said to \emph{(jointly) distinguish supports} if the following implication holds, for any $k,\ell\in\cK$:
\[
\phi_i\inv(\supp(k))=\phi_i\inv(\supp(\ell)) \text{ for all i} \quad \Rightarrow \quad \supp(k)=\supp(\ell).
\]
In view of \Cref{Rem:cons-SpcK}, this is equivalent to the subset~$\cup_i\Img(\phi_i)$ intersecting non-trivially every non-empty $C(k,\ell)$, that is, to the purely topological condition that the subset~$\cup_i\Img(\phi_i)$ be patch-dense in~$\Spc(\cK)$.
\end{Def}

\begin{Exa}
Let $D\subseteq \SpcK$ be a subspace. Consider on~$D$ the restricted support theory for~$\cK$, defined by~$\supp_D(k)=D\cap \supp(k)$ for every object~$k\in\cK$.
Then $D$ is patch-dense if and only if the map $\{\cJ\subseteq\cK\textrm{ radical tt-ideal}\}\too \textrm{Subsets}(D)$ mapping~$\cJ$ to $\cup_{k\in\cJ}\supp_D(k)$ is injective. (Recall that $\cJ$ radical means~$k\potimes{s}\in\cJ\Rightarrow k\in\cJ$; this condition is automatic if~$\cK$ is rigid.)
Indeed, distinguishing supports means distinguishing principal (radical) tt-ideals~$\ideal{k}$, and the latter characterize all (radical) tt-ideals since $\cJ=\cup_{k\in\cJ}\ideal{k}$.
\end{Exa}

\begin{Lem}
\label{Lem:dense-epi}%
Let $X$ be a spectral space. Let $D$ be a set and $i\colon D\to X$ a function (\eg the inclusion of a subset). Then the following are equivalent:
\begin{enumerate}[label=\rm(\roman*), ref=\rm(\roman*)]
  \item\label{it:dense-epi.dense} The image $i(D)$ is patch-dense in~$X$.
  \item\label{it:dense-epi.UV} For every quasi-compact open~$U,V\subseteq X$, if $i\inv(U)\subseteq i\inv(V)$ then $U\subseteq V$.
  \item\label{it:dense-epi.epi} For every pair of spectral maps $\alpha,\beta\colon X\to Z$ to a spectral space~$Z$, if $\alpha\,i=\beta\,i$ then $\alpha=\beta$.
\end{enumerate}
\end{Lem}
\begin{proof}
Note that \ref{it:dense-epi.UV} is simply a reformulation of patch-density~\ref{it:dense-epi.dense}: The constructible subsets $U\cap V^c$ form a basis of the constructible topology and \ref{it:dense-epi.UV} says that such $U\cap V^c$ can only avoid~$i(D)$ when $U\cap V^c$ is empty.

\ref{it:dense-epi.dense}$\Rightarrow$\ref{it:dense-epi.epi} is easy. The spectral maps $\alpha,\beta\colon X\to Z$ are also continuous for the constructible topologies, which are Hausdorff. Hence if $\alpha$ and~$\beta$ agree on the dense subset $i(D)$ they are equal.

For \ref{it:dense-epi.epi}$\Rightarrow$\ref{it:dense-epi.UV}, let $U,V\subseteq X$ be two quasi-compact open such that $U\not\subseteq V$ and let us show that $i\inv(U)\not\subseteq i\inv(V)$.
Consider the Sierpi\'{n}ski space $Z=\{0\leadsto 1\}$ with $1$ the closed point and define maps $\alpha,\beta\colon X\to Z$ by asking that $\alpha$ sends $U\cup V$ to $0$ and the rest to~$1$, whereas $\beta$ sends $V$ to~$0$ and the rest to~$1$. These are continuous spectral maps.
The maps $\alpha$ and $\beta$ are different on the subset~$U\cap V^c\neq\varnothing$. By~\ref{it:dense-epi.epi}, $\alpha\,i(x)\neq\beta\,i(x)$ for some~$x\in D$. Hence the element~$i(x)$ belongs to $U\cap V^c$, showing that~$i\inv(U)\ni x$ is not contained in~$i\inv(V)\not\ni x$.
\end{proof}

We can reconstruct a spectral space from a patch-dense subset together with the trace of the quasi-compact opens of the ambient space.
\begin{Hyp}
\label{Hyp:XU}%
Let $X$ be a set with a chosen non-empty collection $\calU$ of subsets, closed under finite intersections and finite unions. (In particular, $\calU$ contains~$X$ and~$\varnothing$.)
\end{Hyp}

Let us say that a map $f\colon X\to X'$ to a spectral space~$X'$ is \emph{spectral} (with respect to~$\calU$) if $f\inv(U)\in\calU$ for every quasi-compact open~$U\in\qcO(X')$.
Note that if $X$ is spectral and~$\calU=\qcO(X)$, then $f$ being spectral is the usual definition (including continuity).
The \emph{spectral closure}~$\wX$ of~$X$ with respect to~$\calU$ is the initial spectral map~$f\colon X\to \wX$ to a spectral space~$\wX$; in other words, any other spectral map~$f'\colon X\to X'$ factors as $f'=\bar{f}'\circ f$ for a unique spectral map~$\bar{f}'\colon \wX\to X'$.
\begin{Thm}
\label{Thm:reconstruction}%
Let $(X,\calU)$ be a pair as in \Cref{Hyp:XU}.
\begin{enumerate}
\item
\label{it:recons-a}%
The spectral closure $\wX$ of~$X$ with respect to~$\calU$ does exist (and is unique up to unique homeomorphism compatible with $X\to \wX$, as usual).
\smallbreak
\item
\label{it:recons-b}%
Let $i\colon X\to X'$ be a (set-)\,function to a spectral space~$X'$ (\eg\ an inclusion) and suppose that $\calU=\SET{i\inv(U)}{U\in\qcO(X')}$ is exactly the `restriction' to~$X$ of the quasi-compact opens of~$X'$. Then the spectral closure $\wX$ with respect to~$\calU$ can be realized as the constructible-closure $\overline{\Img(i)}^{\con}$ of the image of~$X$ in~$X'$.
\end{enumerate}
\end{Thm}

\begin{proof}
Part~\eqref{it:recons-a} is standard but we could not locate a precise reference in the literature.
Let $\SpSp'$ be the category of pairs $(X,\calU)$ as in \Cref{Hyp:XU}, with obvious morphisms (functions pulling back chosen subsets to chosen subsets).
There is an obvious functor $\mathcal{Q}\colon\SpSp\to\SpSp'$ sending a spectral space~$X$ to $(X,\qcO(X))$, which is moreover fully faithful, by definition of spectral maps.
We prove that $\mathcal{Q}$ admits a left adjoint $\mathcal{P}$ thus proving~\eqref{it:recons-a} with $\wX:=\mathcal{P}(X,\calU)$.
Observe that $\SpSp$ is locally small, well-powered~\cite[Corollary~5.2.8]{dst:spectral-spaces} and has a cogenerator~\cite[Corollary~1.2.8]{dst:spectral-spaces} and that $\SpSp'$ is locally small.
In both $\SpSp$ and $\SpSp'$, limits are computed by taking limits in~$\mathsf{Sets}$ and adding the obvious structure. It follows easily that the functor~$\mathcal{Q}$ preserves limits.
By the Special Adjoint Functor Theorem~\cite[Theorem~V.8.2]{MacLane:cats-working} this functor has a left adjoint~$\mathcal{P}$.

For~\eqref{it:recons-b},
let $Y=\overline{\Img(i)}^{\con}$ be the constructible-closure of~$X$ in~$X'$ and $j\colon X\to Y$ the function~$i\colon X\to X'$ corestricted. By~\cite[Theorem~2.1.3]{dst:spectral-spaces} the subspace~$Y$ of~$X$ is a spectral subspace.
The map~$j^*\colon \qcO(Y)\to \calU$ that maps $U$ to~$j\inv(U)$ is surjective by hypothesis on~$\calU$ and the fact that $Y\subseteq X'$ is a spectral subspace.
By \Cref{Lem:dense-epi}\,\ref{it:dense-epi.dense}$\To$\ref{it:dense-epi.UV} this map $j^*\colon \qcO(Y)\to \calU$ is also injective. So $j^*$ is bijective.
On the other hand, let $\eta\colon X\to \wX$ be the unit of the $\mathcal{P}\adj \mathcal{Q}$ adjunction. For every spectral space~$Z$, precomposing with~$\eta$ defines the bijection of the adjunction $\SpSp(\wX,Z)\isoto\SpSp'(X,\mathcal{Q}Z)$. We can therefore apply \Cref{Lem:dense-epi}\,\ref{it:dense-epi.epi}$\To$\ref{it:dense-epi.UV} to the function $\eta\colon X\to \wX$ to see that $\eta^*\colon \qcO(\wX)\to \calU$, defined by~$U\mapsto \eta\inv(U)$, is injective.
Finally, by the universal property applied to~$j\colon X\to Y$, there exists a spectral map~$\bar{j}\colon \wX\to Y$ such that $\bar{j}\circ\eta=j$.
We can again consider $\bar{j}^*\colon \qcO(Y)\to \qcO(\wX)$ for this third map, still defined by~$U\mapsto \bar{j}\inv(U)$, and we get a commutative diagram of lattices (for inclusion):
\[
\xymatrix@R=1em{\qcO(Y) \ar[rr]^-{\bar{j}^*} \ar[rd]_-{j^*}
&& \qcO(\wX) \ar[ld]^-{\eta^*}
\\
& \calU
}
\]
We have shown that $j^*$ is bijective and that $\eta^*$ is injective. Hence $\bar{j}^*$ is bijective.
It follows from Stone duality~\cite[3.2.10]{dst:spectral-spaces} that $\bar{j}$ is a homeomorphism.
\end{proof}

\begin{Rem}
One can also give an explicit formula for~$\wX$ as the constructible-closure of the image of the spectral map $\ev_X\colon X\to\prod_{U\in\calU}\{0\leadsto 1\}$ that is defined by~$(\ev_X(x))(U)=0$ if and only if $x\in U$. Indeed, that product provides a spectral space~$X'$ as in~\Cref{Thm:reconstruction}\,\eqref{it:recons-b}.
\end{Rem}

A direct consequence of \Cref{Thm:reconstruction}\,\eqref{it:recons-b} in tt-geometry is the following:
\begin{Cor}
\label{Cor:reconstruction}%
Let $\cK$ be an essentially small tt-category and let~$X\subseteq \SpcK$ be a patch-dense subset. Then the space $\SpcK$ can be reconstructed as the spectral closure~$\wX$ of the pair~$(X,\calU)$ for $\calU=\SET{X\cap \supp(k)^c}{k\in\cK}$. In particular, $\SpcK$ can be reconstructed from the restricted support theory $(X,\supp_{X})$, where $\supp_{X}(k)=X\cap \supp(k)$ for every~$k\in\cK$.
\qed
\end{Cor}

%------------------------------------------------------------------------------

\section{Detecting tensor-nilpotence}
\label{sec:detection}%

\begin{Rec}
\label{Rec:tensor-nilpotence}
\begin{enumerate}
\item A morphism $f\colon k\to \ell$ in a tensor category is said to be \emph{$\otimes$-nilpotent} if $f\potimes{s}\colon k\potimes{s}\to \ell\potimes{s}$ is zero for $s\gg 0$.
\item More generally $f$ is \emph{$\otimes$-nilpotent on an object~$m$} if $f\potimes{s}\otimes m\colon k\potimes{s}\otimes m\to \ell\potimes{s}\otimes m$ is zero for $s\gg 0$. This explains the meaning of $f$ being \emph{$\otimes$-nilpotent on its cone}.
\item
A family of tensor functors $\{F_i\colon \cK\to \cL_i\}_{i\in I}$ is said to \emph{(jointly) detect $\otimes$-nilpotence} (on a class of morphisms) if for each morphism $f\colon k\to \ell$ (in that class), we have the following implication:
\[
F_i(f)\textrm{ is $\otimes$-nilpotent for all }i\in I\quad \Rightarrow\quad f\textrm{ is $\otimes$-nilpotent.}
\]
\end{enumerate}
\end{Rec}

\begin{Thm}
\label{Thm:dense-patch}%
Let $\cK$ be an essentially small rigid tt-category and let \break $\{F_i\colon \cK\to \cL_i\}_{i\in I}$ be a family of tt-functors.
The following are equivalent:
\begin{enumerate}[label=\rm(\roman*), ref=\rm(\roman*)]
\item \label{it:dense-patch.nilpotent}The tt-functors $F_i\colon \cK\to\cL_i$ jointly detect $\otimes$-nilpotence of morphisms that are $\otimes$-nilpotent on their cones.
\item The maps $\Spc(F_i)\colon \Spc(\cL_i)\to\Spc(\cK)$ jointly distinguish supports.\label{it:dense-patch.supp}
\item The maps $\Spc(F_i)\colon \Spc(\cL_i)\to\Spc(\cK)$ are jointly epimorphic in the category of spectral spaces.\label{it:dense-patch.epi}
\item \label{it:dense-patch.dense}The subset $\cup_{i\in I}\Img(\Spc(F_i))\subseteq\SpcK$ is patch-dense.
\end{enumerate}
\end{Thm}

\begin{Rem}
\label{Rmk:surjective-singleton-family}
In~\cite[Theorem~1.3]{balmer:surjectivity}, a special case of \Cref{Thm:dense-patch} was proved, namely when the family consists of a single tt-functor~$F\colon \cK\to \cL$; in that case, the map~$\Spc(F)$ is surjective.
Indeed, the image of a single spectral map is always a closed subset for the constructible topology~\cite[Corollary~1.3.23]{dst:spectral-spaces}.
\end{Rem}

Let us do a little preparation for the proof of \Cref{Thm:dense-patch}.
\begin{Rem}
\label{Rem:sss}%
Let $\cL$ be an essentially small tt-category, that is not assumed to be rigid.
For every rigid object~$k$ in~$\cL$, we consider $k^\vee$ the dual of~$k$ and we denote by $A_k$ the ring object~$k\otimes k^\vee\cong\ihom(k,k)$, by $\eta_k\colon \unit\to A_k$ its unit (\aka coevaluation) and $\xi_k\colon J_k\to \unit$ its homotopy fiber, so that we have an exact triangle in~$\cL$,
\begin{equation}\label{eq:sss}%
\xymatrix{J_k \ar[r]^-{\xi_k} & \unit \ar[r]^-{\eta_k} & A_k \ar[r]^-{} & \Sigma J_k\,.}
\end{equation}
Let us recall a few standard observations from~\cite{balmer:sss}.
\begin{enumerate}
\item
\label{it:nil-A}%
By the unit-counit relation in the adjunction~$(k\otimes-)\adj (k^\vee\otimes-)$ the map $\eta_k\otimes k$ is a split monomorphism, hence $k$ is a summand of~$k\otimes k^\vee\otimes k$ and $\xi_k\otimes k=0$.
\item
\label{it:nil-locus}%
The tt-ideal generated by~$k$, or equivalently by~$A_k$, is also the nilpotence locus of~$\xi_k$, that is,
$\ideal{k}=\ideal{A_k}=\SET{m\in\cL}{\exists\,s\gg1\textrm{ s.t.\ }\xi_k\potimes{s}\otimes m=0}$.
\item
Any tt-functor preserves rigidity, duals, $\eta_k$, etc. So once Construction~\eqref{eq:sss} is performed in one category (\eg\ the $\cK$ in the theorem), its properties get transported by any tt-functor $F\colon \cK\to \cL$ (\eg\ the $F_i$), even if $\cL$ is not assumed rigid.
\end{enumerate}

Let us improve on~\eqref{it:nil-locus}:
\begin{enumerate}
\setcounter{enumi}{3}
\item\label{it:nil}%
Let $\ell\in\cL$ be another rigid object and consider the map~$f=\xi_k\otimes \ell\colon J_k\otimes \ell\to \ell$.
Then $f$ is $\otimes$-nilpotent (in fact zero) on its cone, which is~$A_k\otimes\ell$, because of~\eqref{it:nil-A}.
Furthermore, we have
\[
f\textrm{ is $\otimes$-nilpotent if and only if }\ell\in\ideal{k}.
\]
Indeed, if $0=f\potimes{s}\colon J_k\potimes{s}\otimes \ell\potimes{s}\to \ell\potimes{s}$, we see that $\ell\potimes{s}\in\ideal{\cone(f\potimes{s})}\subseteq\ideal{\cone(f)}\subseteq\ideal{\cone(\xi_k)}=\ideal{A_k}=\ideal{k}$ and therefore $\ell\in\ideal{k}$, since $\ell$ is rigid.
Conversely, if~$\ell\in\ideal{k}=\ideal{A_k}$ we already know that $\xi_k$ is nilpotent on~$\ell$ by~\eqref{it:nil-locus}.
This proves the claim. By rigidity again, the condition $\ell\in\ideal{k}$ is equivalent to~$\supp(\ell)\subseteq\supp(k)$.
\end{enumerate}
\end{Rem}

\begin{proof}[Proof of \Cref{Thm:dense-patch}]
We abbreviate $Y_i=\Spc(\cL_i)$ and $X=\Spc(\cK)$ with $\phi_i=\Spc(F_i)\colon Y_i\to X$.
The equivalence between \ref{it:dense-patch.supp} and \ref{it:dense-patch.dense} was already explained in \Cref{Def:separate-supports}.
The equivalence between \ref{it:dense-patch.epi} and \ref{it:dense-patch.dense} follows easily from \Cref{Lem:dense-epi} applied to~$D=\cup_i\Img(\phi_i)$.
It suffices to prove that \ref{it:dense-patch.nilpotent} is equivalent to~\ref{it:dense-patch.supp}.

For \ref{it:dense-patch.nilpotent}$\Rightarrow$\ref{it:dense-patch.supp}, let $k,\ell\in\cK$ such that $\supp(\ell)\not\subseteq\supp(k)$. Let $f=\xi_k\otimes \ell$ as in \Cref{Rem:sss}\,\eqref{it:nil} which tells us that $f$ is not $\otimes$-nilpotent, although it is $\otimes$-nilpotent on its cone.
By our assumption~\ref{it:dense-patch.nilpotent}, there must exist some $i\in I$ such that $F_i(f)$ is not $\otimes$-nilpotent. Let $k_i=F_i(k)$ and $\ell_i=F_i(\ell)$ in~$\cL_i$, which are rigid objects, with $k_i^\vee\cong F_i(k^\vee)$ and~$\ell_i^\vee\cong F_i(\ell^\vee)$. Under these identifications, $\xi_{k_i}=F_i(\xi_k)$ and $F_i(f)=\xi_{k_i}\otimes \ell_i$. We can thus apply \Cref{Rem:sss}\,\eqref{it:nil} to~$k_i$ and~$\ell_i$ in~$\cL_i$ and deduce that $\phi_i\inv(\supp(\ell))=\supp(\ell_i)\not\subseteq\supp(k_i)=\phi_i\inv(\supp(k))$.
This shows that the family $(\phi_i)_{i\in I}$ jointly distinguishes supports.

The proof of \ref{it:dense-patch.supp}$\Rightarrow$\ref{it:dense-patch.nilpotent} is a straightforward adaptation of the proof in~\cite[Theorem~1.4]{balmer:surjectivity}.
Indeed, if $f\colon k\to \ell$ is a morphism in $\cK$ that is $\otimes$-nilpotent on its cone and $F_i(f)$ is $\otimes$-nilpotent, say $F_i(f\potimes{s_i})=0$, then as in \loccit{} we deduce that
\[
\phi_i\inv(\supp(\cone(f\potimes{s_i})))=\phi_i\inv\left(\supp(k\potimes{s_i})\cup\supp(\ell\potimes{s_i})\right).
\]
We now observe that these supports do not change if we replace~$s_i$ by~$1$; for the left-hand side see~\cite[Proposition~2.10]{balmer:surjectivity}.
In other words, we have, for all~$i\in I$:
\[
\phi_i\inv(\supp(\cone(f)))=\phi_i\inv(\supp(k)\cup\supp(\ell))=\phi_i\inv(\supp(k\oplus \ell)).
\]
By our assumption we get
\[
\supp(\cone(f))=\supp(k\oplus \ell)=\supp(k)\cup\supp(\ell)
\]
and one concludes that $f$ is $\otimes$-nilpotent by a standard argument: $k,\ell\in\ideal{\cone(f)}$ forces~$f\colon k\to \ell$ to be $\otimes$-nilpotent on~$k$ (and~$\ell$) hence to be $\otimes$-nilpotent.
\end{proof}

% ------------------------------------------------------------------------------

\section{Examples}
\label{sec:examples}%

Let us show that patch-dense subsets commonly arise in nature.

% ------------------------------------------------------------------------------

\subsection{Visible locus}
\label{sec:visibility}

\begin{Exa}
\label{Exa:SH}
Consider the tt-category of finite $p$-local spectra $\SH^c_{(p)}$ for some prime~$p$.
This example was already discussed in the introduction so we will be brief and content ourselves with providing references.
A morphism $f\colon k\to \ell$ in $\SH^c_{(p)}$ is $\otimes$-nilpotent if and only if $K(n)_*(f)$ is for all $\infty>n\geq 0$, where $K(n)$ denotes the $n$th Morava $K$-theory (at the prime~$p$).
This follows from the Nilpotence Theorem of Devinatz, Hopkins, Smith in the form of~\cite[Theorem~3.iii), Corollary~2.2]{hopkins-smith:chromatic}.
This corresponds to the fact that the `finite' points in the spectrum are patch-dense, see~\eqref{eq:Spc-SH}.

We turn this into a more general observation, using that the `finite' points are precisely the weakly visible ones:
\end{Exa}

\begin{Rec}
\label{Rec:weakly-visible}
Recall that a subspace $Y\subseteq X$ of a spectral space~$X$ is called \emph{weakly visible} if $Y=V\cap W^{c}$ for two Thomason subsets $V,W\subseteq X$.
If $\cT$ is a rigidly-compactly generated tt-category, then the weakly visible subsets in~$\Spc(\cT^c)$ are those that can be described in terms of $\otimes$-idempotents in~$\cT$, as in~\cite{balmer-favi:idempotents}.
For a theory of stratification of~$\cT$ based on this the reader can consult~\cite{barthel-heard-sanders:stratification-Mackey}.
\end{Rec}

\begin{Prop}
\label{Prop:weakly-visible}
Let $X$ be a spectral space.
The set of weakly visible points in~$X$ is patch-dense.
\end{Prop}
\begin{proof}
Let $W=U\cap V^c$ be a non-empty basic open for the constructible topology, with $U,V$ quasi-compact open in~$X$ and let $x\in W$.
Since the closed subspace $V^c$ is itself a spectral space there is a point $y\in V^c$ with $y\leadsto x$ and which has no further generalizations in~$V^c$~\cite[Corollary~4.1.4]{dst:spectral-spaces}.
As $U$ is open we have $y\in W$ so it suffices to show that $y$ is weakly visible.
By construction, $\{y\}=V^c\cap\mathrm{gen}\{y\}$, where $\mathrm{gen}$ denotes the set of generalizations.
We conclude since $V^c$ is Thomason and the set $\mathrm{gen}\{y\}=\cap_{O \in\qcO(X)\mid y\in O}O$ is the complement of a Thomason.
\end{proof}

\begin{Exa}
\label{Exa:punctured}
For an algebraic situation similar to \Cref{Exa:SH}, let $R$ be a local commutative ring whose maximal ideal~$\gm$ is not finitely generated, not even up to radicals.
(For example, the local ring at a closed point in infinite affine space.)
This is equivalent to the closed point~$\{\gm\}$ not being (weakly) visible in~$\Spec(R)$ in the sense of \Cref{Rec:weakly-visible}.
By \Cref{Prop:weakly-visible}, the punctured spectrum $\Spec(R)\sminus\{\gm\}$ is patch-dense.
This corresponds to the fact that $\otimes$-nilpotence of maps in $\Dperf(R)$ is detected by the residue field functors $-\otimes_R\kappa(\gp)$ for the non-maximal primes~$\gp\in\Spec(R)\sminus\{\gm\}$.
\end{Exa}

\begin{Rem}
\label{Rem:noetherian}%
In both \Cref{Exa:SH,Exa:punctured}, the complement of the closed point(s) was patch-dense.
This is not possible if the space $X$ is noetherian.
Indeed, when $X$ is noetherian, the open complement $X\sminus \{x\}$ of every closed point~$x$ is quasi-compact, as every open is; this implies that~$\{x\}$ is open for the constructible topology.
Hence any patch-dense subset must contain all the closed points.

In fact, for noetherian spaces there is a \emph{smallest} patch-dense subset.
It consists precisely of the locally closed points~\cite[Proposition~4.5.21, Corollary~8.1.19]{dst:spectral-spaces}.
\end{Rem}

\begin{Exa}
\label{Exa:Jacobson}%
A (spectral) space~$X$ is called \emph{Jacobson} if its closed points are dense in every closed subset.
For instance, the underlying space of every scheme locally of finite type over a field, or locally of finite type over~$\mathbb{Z}$, is a Jacobson space.

In a Jacobson space, a subset~$D\subseteq X$ that contains all closed points is necessarily patch-dense.
Indeed, let $Y^c\cap Z\neq\varnothing$ be a non-trivial basic constructible, with $Y$ and~$Z$ closed subsets with quasi-compact complement. Since the open $Y^c\cap Z$ of~$Z$ is non-empty, it must contain a closed point of~$X$, which is in~$D$ by assumption.

In view of \Cref{Rem:noetherian}, in `usual' algebraic geometry, say, when dealing with schemes of finite type over a field, or of finite type over~$\mathbb{Z}$, a subset is patch-dense if and only if it contains all the closed points.
\end{Exa}

% ------------------------------------------------------------------------------

\subsection{Retractable limits}
\label{sec:limits}

\begin{Not}
\label{Not:colimit-tt-cats}
Let $\cK=\colim \cK_i$ be the directed colimit of essentially small tt-categories.
Let $X_i:=\Spc(\cK_i)$ and $X=\Spc(\cK)$.
We denote by $\pi_i^*\colon \cK_i\to \cK$ the canonical tt-functor, and by $\pi_i\colon X\to X_i$ the induced map on spectra.
Recall from~\cite[Proposition~8.2]{gallauer:tt-fmod} that the maps~$\pi_{i}$ induce a homeomorphism
\begin{equation}
\label{eq:limit}
X\isoto\lim_i X_i.
\end{equation}
Let us assume that each $\pi_i^*$ admits a tt-retraction
\[
\sigma_i^*\colon \cK\to\cK_i,
\]
so that $\pi_i\circ\sigma_i=\id\colon X_i\to X_i$ for all~$i$, where $\sigma_i:=\Spc(\pi_i^*)$.
\end{Not}

\begin{Cor}
\label{Cor:dense-limit}%
The family $\{\sigma_i^*\colon\cK\to\cK_i\}$ jointly detects $\otimes$-nilpotence.
In particular, if $\cK$ is rigid then the subset $\cup_i\Img(\sigma_i)\subseteq\Spc(\cK)$ is patch-dense.
\end{Cor}
\begin{proof}
Let $f\colon k\to \ell$ be a morphism in $\cK$ such that $\sigma_i^*(f)$ is $\otimes$-nilpotent for all~$i$.
We may choose $i$ such that $f=\pi_i^*(f')$ for some $f'\colon k'\to \ell'$ in $\cK_i$.
But then $f'=\sigma_i^*\circ\pi_i^*(f')=\sigma_i^*(f)$ is $\otimes$-nilpotent hence so is $f$.
The second statement follows from \Cref{Thm:dense-patch}.
\end{proof}

We now discuss some examples of this result.

\subsection{Profinite equivariance}

\begin{Exa}
\label{Exa:SH(G)}
Let $G$ be a profinite group.
The tt-category of finite genuine $G$-spectra is the directed colimit
\[
\SH(G)^c=\colim_N\SH(G/N)^c
\]
where $N\normaleq G$ runs through open normal subgroups, and the transition is given by inflation functors, see~\cite{balchin2024profiniteequivariantspectratensortriangular}.
For any such $N\normaleq G$, consider the geometric fixed points functor $\Phi^N\colon\SH(G)^c\to\SH(G/N)^c$ which gives a retraction to inflation.
It follows from (the easy part of) \Cref{Cor:dense-limit} that the geometric fixed points functors jointly detect $\otimes$-nilpotence.\footnote{\,In fact, it is known~\cite[Corollary~8.7]{balchin2024profiniteequivariantspectratensortriangular} that the geometric fixed points functors $\{\Phi^H\}_{H\leq G}$ (running through closed subgroups $H$) jointly detect $\otimes$-nilpotence for morphisms $f\colon k\to L$ of genuine $G$-spectra in which only $k$ is assumed compact.}
The geometric fixed points for closed subgroups induce a bijection~\cite[Proposition~7.4]{balchin2024profiniteequivariantspectratensortriangular}
\[
\subname(G)/G\times\Spc(\SH^c)\xto{\sim}\Spc(\SH(G)^c).
\]
On the other hand, we deduce from the second part of \Cref{Cor:dense-limit} that the subset
\[
\subname^{\textup{open}}(G)/G\times\Spc(\SH^c) \subseteq \Spc(\SH(G)^c)
\]
corresponding to \emph{open} subgroups is already patch-dense.
(One could combine this with \Cref{Exa:SH} to exhibit an even smaller patch-dense subset.)
Note that this subset is directly linked to the tt-geometry of equivariant spectra for \emph{finite} groups.
\end{Exa}

\begin{Exa}
\label{Exa:DPerm}
We continue to denote by~$G$ a profinite group.
Let $k$ be a field of characteristic~$p$.
The tt-category of compact derived permutation modules is the directed colimit
\[
\DPerm(G;k)^c=\colim_{N\normaleq G}\DPerm(G/N;k)^c
\]
along the inflation functors, see~\cite{balmer-gallauer:Dperm} or~\cite{MR4866349}.
In this case, it is the \emph{modular fixed points} $\Psi^N\colon \DPerm(G)\to\DPerm(G/N)$ of~\cite{balmer-gallauer:TTG-Perm} that yield retractions to inflation.
The spectrum admits a set-theoretic stratification
\begin{equation}
\label{eq:Spc-DPerm}
\Spc(\DPerm(G)^c)=\coprod_{H\leq G}\Spc(\Db(k(\WGH)))
\end{equation}
where $\WGH=N_G(H)/H$ is the Weyl group.
It follows again from \Cref{Cor:dense-limit} that a patch-dense subset is given by the strata in~\eqref{eq:Spc-DPerm} for $H\leq G$ \emph{open}.
\end{Exa}

\subsection{Support varieties}
\label{sec:support-varieties}

\begin{Exa}
Let $G$ again be a profinite group, and $k$ a field of characteristic~$p$.
The bounded derived category of finite-dimensional $k$-linear (discrete) $G$-representations is the directed colimit
\[
\Db(\mathrm{mod}(G;k))=\colim_{N\normaleq G}\Db(\mathrm{mod}(G/N;k))
\]
along inflation.
It follows easily from the case of finite groups~\cite{benson-carlson-rickard:tt-class-stab(kG)} that the spectrum is an invariant of the cohomology algebra:
\begin{equation}
\label{eq:Spc-Db}
\Spc(\Db(\mathrm{mod}(G;k)))=\Spech(\Hm^{\bullet}(G;k)),
\end{equation}
see~\cite[Proposition~6.5]{gallauer:tt-dtm-algclosed}.
Here we do not have a retraction $\Db(\mathrm{mod}(G;k))\to\Db(\mathrm{mod}(G/N;k))$ to inflation in general so this does not fit in the framework of \Cref{sec:limits}.
Nevertheless, there is another natural family of functors that jointly detects $\otimes$-nilpotence.
At the level of cohomology algebras the statement is that the ring maps
\[
\Res_E\colon\Hm^\bullet(G;k)\to\Hm^\bullet(E;k)
\]
for \emph{finite} elementary abelian $p$-subgroups $E\leq G$ detects which elements are nilpotent, see~\cite[Proposition~8.7]{MR1411658} or~\cite[Theorem~1]{MR2051117}.
We now upgrade this to a categorical statement.
\end{Exa}

\begin{Prop}
Let $G$ be a profinite group. The family of restriction functors
\[
\Res_E\colon\Db(\mathrm{mod}(G;k))\to\Db(\mathrm{mod}(E;k))
\]
to finite elementary abelian $p$-subgroups $E\leq G$ jointly detects $\otimes$-nilpotence.
Hence their images on spectra cover a patch-dense subset
\[
\cup_E\Img(\Spc(\Res_E))\subseteq\Spc(\Db(\mathrm{mod}(G;k))).
\]
\end{Prop}
\begin{proof}
Let $f$ be a morphism in $\Db(\mathrm{mod}(G;k))$ such that $\Res_E(f)$ is $\otimes$-nilpotent for all finite elementary abelian $p$-subgroups $E\leq G$.
There is an open normal subgroup $N\normaleq G$ such that $f$ is inflated from some $g$ in~$\Db(\mathrm{mod}(G/N;k))$.
By~\cite[Proposition~1]{MR2051117}, there exists $N'\normaleq N$ another open normal subgroup such that for every elementary abelian $p$-subgroup $F'\leq G/N'$ its image $F:=F'N/N$ in $G/N$ is the image of a finite elementary abelian $p$-subgroup~$E$ of~$G$, that is, $F=EN/N$.
For each such $F'$ we have
\begin{equation}
\label{eq:rho-infl}
\Res^{G/N'}_{F'}\Infl^{G/N}_{G/N'}(g)=\Infl_{F'}^{F}\Res^{G/N}_F(g).
\end{equation}
To prove that $f$ is $\otimes$-nilpotent it suffices to show $\Infl^{G/N}_{G/N'}(g)$ is, hence by~\eqref{eq:rho-infl} and Quillen, that $\Res^{G/N}_F(g)=\Res^{G/N}_{EN/N}(g)$ is (for all~$F'$).
While inflation typically is not faithful, it is along surjections between elementary abelian groups because the latter are split.
Hence it suffices to show $\otimes$-nilpotence of
\[
\Infl_{E}^{EN/N}\Res^{G/N}_{EN/N}(g)=\Res_E^G\Infl_{G}^{G/N}(g)=\Res_E^G(f),
\]
which was exactly our assumption.
\end{proof}

\begin{Exa}
Let $G=(C_p)^{\mathbb{N}}$ be a countably infinite pro-elementary abelian group, cf.~\cite[Example~6.10]{MR4866349}.
Then $\Spc(\Db(\mathrm{mod}(G;k)))=\overline{\mathbb{P}}_k^\infty$ is an infinite-dimensional projective space, together with a unique closed point attached.
The patch-dense subset $\cup_E\Img(\Spc(\rho_E))$ is given by $\cup_n\overline{\mathbb{P}}_k^n$, those points with only finitely many `non-zero coordinates'.
\end{Exa}

%------------------------------------------------------------------------------

%\bibliographystyle{alpha}
%\bibliography{ref}

%------------------------------------------------------------------------------

\end{document}